\documentclass[a4paper, 12pt]{amsart}
\usepackage{amsmath, amssymb, amsthm, graphicx, stmaryrd}
\usepackage{mathrsfs}
\usepackage{comment}
\usepackage[all]{xy}
\usepackage{tikz}
\usetikzlibrary{decorations.markings}

\usepackage[top=40mm,bottom=40mm,left=30mm,right=30mm]{geometry}

\theoremstyle{definition}
\newtheorem{theorem}{Theorem}[section]

\newtheorem{lemma}[theorem]{Lemma}
\newtheorem{proposition}[theorem]{Proposition}
\newtheorem{corollary}[theorem]{Corollary}
\newtheorem{remark}[theorem]{\rm Remark}

\makeatletter
\@addtoreset{equation}{section} 
\makeatother

\DeclareMathOperator{\Diff}{Diff}

\DeclareMathOperator{\Ang}{Ang}
\DeclareMathOperator{\id}{id}
\DeclareMathOperator{\vol}{vol}
\DeclareMathOperator{\Symp}{Symp}
\DeclareMathOperator{\rel}{rel}

\DeclareMathOperator{\Ker}{Ker}
\DeclareMathOperator{\Cal}{Cal}
\DeclareMathOperator{\Hom}{Hom}

\DeclareMathOperator{\rot}{rot}

\title[Extensions of quasi-morphisms to the symplectomorphism group]{Extensions of quasi-morphisms to the
symplectomorphism group of the disk}
\author{Shuhei Maruyama}
\address{Graduate School of Mathematics,
Nagoya University, Japan}
\email{m17037h@math.nagoya-u.ac.jp}

\begin{document}

\begin{abstract}
  On the group $\Symp(D, \partial D)$ of
  symplectomorphisms of the disk which are the identity
  near the boundary, there are homogeneous quasi-morphisms called
  the Ruelle invariant and Gambaudo-Ghys quasi-morphisms.
  In this paper, we show that
  the above homogeneous quasi-morphisms
  extend to homogeneous quasi-morphisms on the whole group
  $\Symp(D)$ of symplectomorphisms of the disk.
  As a corollary, we show that the second bounded cohomology
  $H_b^2(\Symp(D))$ is infinite-dimensional.
\end{abstract}

\maketitle

\section{Introduction}

Let $G$ be a group.
A function $\phi:G \to \mathbb{R}$ is called a
{\it quasi-morphism} if there exists a constant $C$ such that
the condition $|\phi(gh) - \phi(g) - \phi(h)| \leq C$
holds for any $g, h \in G$.
A quasi-morphism $\phi$ is {\it homogeneous} if,
for any integer $n \in \mathbb{Z}$ and any $g \in G$,
the condition $\phi(g^n) = n\phi(g)$ holds.
Let $Q(G)$ denote the vector space of all homogeneous
quasi-morphisms on $G$.
The {\it homogenization} $\overline{\phi}$
of a quasi-morphism $\phi$
is defined by
$\overline{\phi}(g) = \lim_{n \to \infty} \phi(g^n)/n$.
This homogenization $\overline{\phi}$ is a homogeneous
quasi-morphism.

Let $K$ be a subgroup of $G$.
It is natural to ask whether given homogeneous quasi-morphism
$\psi:K \to \mathbb{R}$ can be extended to a homogeneous
quasi-morphism on $G$.
This extension problem of quasi-morphisms is studied in some
papers
(\cite{shtern15}, \cite{kawasaki18}, \cite{kawasaki_kimura19}).
In this paper, we consider the extension problem of the
Ruelle invariant and Gambaudo-Ghys quasi-morphisms on
$\Symp(D, \partial D)$ to
the group $\Symp(D)$.

Let $D = \{ (x, y) \in \mathbb{R}^2 \mid x^2 + y^2 \leq 1 \}$
be a unit disk and $\omega = dx \wedge dy$ a standard
symplectic form.
Let $\Symp(D)$ denote the symplectomorphism group
of the disk $D$ and
$\Symp(D, \partial D)$ the subgroup of $\Symp(D)$
consisting of symplectomorphisms which are
the identity near the boundary $\partial D$.
There are many quasi-morphisms on $\Symp(D, \partial D)$.
For example, in \cite{ruelle85}, Ruelle constructed a homogeneous
quasi-morphism on $\Symp(D, \partial D)$,
which is called the Ruelle invariant.
In \cite{entov_polterovich03} and \cite{gambaudo_ghys04},
it was shown that the vector space
$Q(\Symp(D, \partial D))$
of homogeneous quasi-morphisms on $\Symp(D, \partial D)$
is infinite-dimensional.
In \cite{gambaudo_ghys04},
Gambaudo and Ghys constructed
countably many quasi-morphisms on $\Symp(D, \partial D)$
by integrating on the disk the signature quasi-morphism on
pure braid group $P_n$ on $n$-strands.
Brandenbursky\cite{brandenbursky11} generalized this idea
to any quasi-morphisms on $P_n$
and defined the linear map
$\Gamma_n:Q(P_n) \to Q(\Symp(D, \partial D))$.
We call a homogeneous quasi-morphism in $\Gamma_n(Q(P_n))$
a {\it Gambaudo-Ghys quasi-morphism}.

The well-definedness of the Ruelle invariant and
Gambaudo-Ghys quasi-morphisms comes from
the fact that the group $\Symp(D, \partial D)$ is contractible.
Since the group $\Symp(D)$ is not contractible,
both constructions cannot be applied to the group $\Symp(D)$.
However, we show that the following theorem holds.
\begin{theorem}
  The Ruelle invariant and Gambaudo-Ghys quasi-morphisms
  on the group $\Symp(D, \partial D)$ extend
  to homogeneous quasi-morphisms on the group $\Symp(D)$.
  In particular, the vector space $Q(\Symp(D))$ is infinite-dimensional.
\end{theorem}

Let $B_n$ denote the braid group on $n$-strands.
Ishida\cite{ishida14} showed that the restriction
$\Gamma_n|_{Q(B_n)}:Q(B_n) \to Q(\Symp(D, \partial D))$ and the induced map
$EH_b^2(B_n) \to EH_b^2(\Symp(D, \partial D))$ are injective.
Here the symbol $EH_b^2(\cdot)$ denotes the second
exact bounded cohomology defined in section $2$.
Together with this Ishida's theorem and the fact that
$EH_b^2(B_n)$ is infinite-dimensional\cite{bestvina_fujiwara02},
we obtain the following corollary.
\begin{corollary}
  The exact bounded cohomology $EH_b^2(\Symp(D))$ and therefore the
  second bounded cohomology $H_b^2(\Symp(D))$ are
  infinite-dimensional.
\end{corollary}

This paper is organized as follows.
In section $2$, we recall the bounded cohomology.
In section $3$, we consider the Ruelle's construction and
Gambaudo-Ghys's construction on the universal covering of the
group $\Symp(D)$.
In section $4$, we show the main theorem.
In section $5$, we deal with the extension problem
of homomorphisms,
which relate to the Calabi invariant.

\subsection*{Acknowledgements}
The author thanks to Morimichi Kawasaki and Masakazu Tsuda
for useful discussions.

\section{Bounded cohomology}
Let $G$ be a group.
Let $C^p(G)$ denote the set of maps
from $p$-fold product $G^p$ to $\mathbb{R}$
for $p>0$ and let $C^0(G) = \mathbb{R}$.
The coboundary operator $\delta : C^p(G) \to C^{p+1}(G)$
is defined by
\begin{align*}
  \delta c (g_1, \dots, g_{p+1})
  = c(g_2, \dots, g_{p+1}) &+ \sum_{i=1}^{p}(-1)^i
  c(g_1, \dots , g_i g_{i+1}, \dots , g_{p+1})\\
  &+ (-1)^{p+1}c(g_1, \dots , g_p)
\end{align*}
and the cohomology of $(C^*(G), \delta)$ is called
the {\it group cohomology of} $G$ and denoted by $H^*(G)$.
Note that, by definition, the first cohomology $H^1(G)$ is
equal to the vector space $\Hom(G,\mathbb{R})$ of homomorphisms
from $G$ to $\mathbb{R}$.

Let $(C_b^*(G),\delta)$ denote the subcomplex of
$(C^*(G), \delta)$ consisting of all bounded functions.
Its cohomology is called the {\it bounded cohomology of} $G$
and denoted by $H_b^*(G)$.
The inclusion $C_b^*(G) \hookrightarrow C^*(G)$
induces the map $H_b^*(G) \to H^*(G)$ called the
{\it comparison map}.
The kernel of the comparison map
is called the {\it exact bounded cohomology of} $G$
and denoted by $EH_b^*(G)$.
Then there is an exact sequence
\[
  0 \to H^1(G) \to Q(G) \to EH_b^2(G) \to 0,
\]
where the map $Q(G) \to EH_b^2(G)$ is given by
$\phi \mapsto [\delta \phi]$.
In other words, the second exact bounded cohomology is isomorphic to
the quotient $Q(G)/H^1(G)$.

\section{Homogeneous quasi-morphisms on $\widetilde{\Symp}(D)$}\label{section:on_G^}
On the group $\Symp(D, \partial D)$, many
homogeneous quasi-morphisms
are constracted(\cite{entov_polterovich03}, \cite{gambaudo_ghys04}).
In this section, we apply to the universal covering $\widetilde{\Symp}(D)$
the methods explained in \cite{gambaudo_ghys04}
to obtain homogeneous quasi-morphisms.

\subsection{Ruelle invariant}\label{subsection:ruelle}
For $x \in D$ and a path $\{ g_t \}_{t \in [0,1]}$ in $\Symp(D)$ with
$g_0 = \id$,
let $u_t(x) \in \mathbb{R}^2 \setminus (0,0)$
denote the first column of $dg_t(x) \in SL(2, \mathbb{R})$.
Then the variation of the angle of $u_t(x)$ depends on $x$ and
the homotopy class of the path $\{ g_t \}_{t \in [0,1]}$
relatively to fixed ends.
Thus, for $\alpha \in \widetilde{\Symp}(D)$
represented by the path $\{ g_t \}_{t \in [0,1]}$, we denote
the variation of the angle of $u_t(x)$
by $\Ang_{\alpha}(x)$.
For $\alpha, \beta \in \widetilde{\Symp}(D)$ and
a path $\{ h_t \}_{t \in [0,1]}$ which represents $\beta$,
the inequality
\begin{equation}\label{ang_inequality}
  | \Ang_{\alpha \beta}(x) - \Ang_{\beta}(x)
  - \Ang_{\alpha}(h_1(x))| < 1/2
\end{equation}
holds for any $x \in D$,
where we consider $S^1$ as $\mathbb{R}/\mathbb{Z}$.
By the above inequality (\ref{ang_inequality}), the
function $r : \widetilde{\Symp}(D) \to \mathbb{R}$ defined by
\begin{equation}
  r(\alpha) = \int_{D} \Ang_{\alpha}\cdot \  \omega
\end{equation}
is a quasi-morphism on $\widetilde{\Symp}(D)$.
Let $\overline{r}$ denote the homogenization of $r$.
By construction, the restriction of $\overline{r}$ to
$\Symp(D, \partial D)
= \widetilde{\Symp(D, \partial D)}$
coincides with the classical Ruelle invariant on the disk.
Since the Ruelle invariant is non-trivial
homogeneous quasi-morphism on $\Symp(D, \partial D)$, so is
$\overline{r} : \widetilde{\Symp}(D) \to \mathbb{R}$.

\subsection{Gambaudo-Ghys construction}
Let $X_n$ denote the $n$-fold configuration space of $D$.
For $\alpha \in \widetilde{\Symp}(D)$ and for almost all
$x = (x_1, \dots, x_n) \in X_n$, a pure braid
$\gamma(\alpha;x) \in P_n$ is defined as follows.
Let us fix a base point $z = (z_1, \dots, z_n) \in X_n$.
Take a path $\{ g_t \}_{t \in [0,1]}$ which represents
$\alpha \in \widetilde{\Symp}(D)$.
Then we obtain a loop $l(\alpha;x)$ in $X_n$ by
\[
  l(\alpha;x) = \begin{cases}
    (1-3t)z + 3tx & (0 \leq t \leq 1/3)\\
    g_t(x) & (1/3 \leq t \leq 2/3)\\
    (3-3t)g_1(x) + (3t-2)z & (2/3 \leq t \leq 1),
\end{cases}
\]
where $g_t(x) = (g_t(x_1), \dots, g_t(x_n)) \in X_n$.
This loop is well-defined for almost all $x \in X_n$ and
its homotopy class is independent of the choice of
representatives of $\alpha$.
Thus we define $\gamma(\alpha;x) \in P_n$ as the braid
represented by the loop $l(\alpha;x)$.

For a (homogeneous) quasi-morphism
$\phi$ on the pure braid group $P_n$,
we define a function
$\hat{\Gamma}_n(\phi) : \widetilde{\Symp}(D) \to \mathbb{R}$
by
\begin{equation}
  \hat{\Gamma}_n(\phi)(\alpha) = \int_{X_n} \phi(\gamma(\alpha;x))dx,
\end{equation}
where $dx$ is the volume form on $X_n$ induced from the
volume form $\omega^n$ on $D^n$.
The proof of the integrability of the function
$\phi(\gamma(\alpha;x))$ and the fact that
$\hat{\Gamma}_n(\phi)$ is a quasi-morphism
is the same as in \cite[Lemma 4.1]{brandenbursky11}.
Let $\widetilde{\Gamma}_n(\phi)$ denote the homogenization of
$\hat{\Gamma}_n(\phi)$, then we have a linear map
$\widetilde{\Gamma}_n: Q(B_n) \to Q(\widetilde{\Symp}(D))$.

Let $\iota : \Symp(D, \partial D)
= \widetilde{\Symp(D, \partial D)} \hookrightarrow \widetilde{\Symp}(D)$
be the inclusion.
Then, by construction, the composition
$\iota^* \circ\widetilde{\Gamma}_n : Q(B_n) \to Q(\Symp(D, \partial D))$
coincides with the original Gambaudo-Ghys construction on
$\Symp(D, \partial D)$.
Put $\Gamma_n = \iota^* \circ\widetilde{\Gamma}_n$.
Ishida \cite{ishida14} showed that
the restriction
\[
  \Gamma_n|_{Q(B_n)} :Q(B_n) \to Q(\Symp(D, \partial D))
\]
is injective.
Thus the linear map
$\widetilde{\Gamma}_n|_{Q(B_n)} : Q(B_n) \to Q(\widetilde{\Symp}(D))$
is also injective.
Since the vector space $Q(B_n)$ is infinite-dimensional
\cite{bestvina_fujiwara02},
the following proposition holds.

\begin{proposition}
  The vector space $Q(\widetilde{\Symp}(D))$ is infinite-dimensional.
\end{proposition}

Ishida showed in \cite{ishida14} that the linear map
$\Gamma_n|_{Q(B_n)} :Q(B_n) \to Q(\Symp(D, \partial D))$
induces the injective map
$\Gamma_n^* : EH_b^2(B_n) \to EH_b^2(\Symp(D, \partial D))$.
Thus we have the injection
$\widetilde{\Gamma}_n^*: EH_b^2(B_n) \to
EH_b^2(\widetilde{\Symp}(D))$.
Since the $EH_b^2(B_n)$ is infinite-dimensional for
$n> 2$ (see \cite{bestvina_fujiwara02}),
the following theorem holds.

\begin{theorem}
  The second exact bounded cohomology $EH_b^2(\widetilde{\Symp}(D))$
  and therefore
  the bounded cohomology $H_b^2(\widetilde{\Symp}(D))$
  are infinite-dimensional.
\end{theorem}

\section{Homogeneous quasi-morphisms on $\Symp(D)$}
In this section, we show that
the Ruelle invariant and the Gambaudo-Ghys quasi-morphisms
on $\Symp(D, \partial D)$
extend to homogeneous quasi-morphisms on $\Symp(D)$.

Let $\phi \in Q(\widetilde{\Symp}(D))$ be either the Ruelle invariant
or a Gambaudo-Ghys quasi-morphism.
Let us consider the short exact sequence
\[
  0 \to \mathbb{Z} = \pi_1(\Symp(D)) \to \widetilde{\Symp}(D)
  \xrightarrow{p} \Symp(D) \to 1.
\]
Since all homogeneous quasi-morphisms on
abelian groups are homomorphism,
the restriction $\phi|_{\pi_1(\Symp(D))}$ is a homomorphism.
Put $a_{\phi} = \phi(1) \in \mathbb{R}$
where $1 \in \pi_1(\Symp(D))$ is the full rotation of the disk $D$.
Let $\widetilde{\Diff}_{+}(S^1)$ be the universal covering of
$\Diff_{+}(S^1)$ the orientation preserving diffeomorphisms of
the circle
and $\rho : \widetilde{\Symp}(D) \to \widetilde{\Diff}_{+}(S^1)$
the restriction to the boundary.
On the group $\widetilde{\Diff}_{+}(S^1)$,
there is a quasi-morphism
$\rot : \widetilde{\Diff}_{+}(S^1) \to \mathbb{R}$
called the rotation number.
Note that the $\rho(1) \in \pi_1(\Diff_{+}(S^1)) = \mathbb{Z}$
is the full rotation of the circle $S^1$ and
thus $\rho^*\rot(1) = \rot(\rho(1)) = 1$.

\begin{remark}
  In general, $a_{\phi}$ is non-zero value.
  For example, let $\phi$ be the Ruelle invariant
  $\overline{r}$ defined in subsection \ref{subsection:ruelle},
  then
  $a_{\phi}$ is equal to the symplectic area of the disk $D$.
\end{remark}

\begin{lemma}\label{lemma:descends}
  The homogeneous quasi-morphism
  $\phi - a_{\phi}\cdot\rho^*\rot$ on $\widetilde{\Symp}(D)$
  descends
  to a homogeneous quasi-morphism on $\Symp(D)$, that is,
  there exists a homogeneous quasi-morphism $\psi$ on
  $\Symp(D)$ satisfying
  $p^* \psi = \phi - a_{\phi}\cdot\rho^*\rot$.
\end{lemma}

\begin{proof}
  By definition of $a_{\phi}$, the homogeneous quasi-morphism
  $\phi - a_{\phi}\cdot\rho^*\rot$ is
  equal to $0$ on $\pi_1(\Symp(D))$.
  Thus, by the following Shtern's theorem, the lemma follows.
\end{proof}

\begin{theorem}\cite{Shtern94}
  Let $1 \to K \to G \xrightarrow{p} H \to 1$
  be a short exact sequence
  and $\phi$ a homogeneous quasi-morphism on $G$.
  If $\phi|_{K} = 0$, then there is a
  homogeneous quasi-morphism $\psi$
  on $H$ such that $\phi = p^*\psi$.
\end{theorem}

\begin{theorem}
  The Ruelle invariant and the Gambaudo-Ghys quasi-morphisms
  on $\Symp(D, \partial D)$
  extend to homogeneous quasi-morphisms on $\Symp(D)$.
  In particular, the vector space $Q(\Symp(D))$
  of homogeneous quasi-morphisms is infinite-dimensional.
\end{theorem}

\begin{proof}
  By definition of $\phi$, we have to prove that the restriction
  $\phi|_{\Symp(D, \partial D)}$ extends to a
  homogeneous quasi-morphism
  on $\Symp(D)$.
  By Lemma \ref{lemma:descends},
  take the homogeneous quasi-morphism $\psi$ on $\Symp(D)$ satisfying
  $p^* \psi = \phi - a_{\phi}\cdot\rho^*\rot$.
  The composition of $p: \widetilde{\Symp}(D) \to \Symp(D)$ and the
  inclusion $\iota : \Symp(D, \partial D)
  \hookrightarrow \widetilde{\Symp}(D)$
  is equal to the inclusion
  $i : \Symp(D, \partial D) \hookrightarrow \Symp(D)$
  and the composition $\rho: \widetilde{\Symp}(D) \to
  \widetilde{\Diff}_{+}(S^1)$
  and $\iota$ is equal to $0$.
  Then we have
  \[
    \phi|_{\Symp(D, \partial D)} =
    \iota^*(p^*\psi - a_\phi \cdot \rho^*\rot)
    =
    \psi|_{\Symp(D, \partial D)}
  \]
  and this implies that the homogeneous
  quasi-morphism $\psi$ is an
  extension of $\phi|_{\Symp(D, \partial D)}$.
\end{proof}

The above theorem implies that
the image $i^*(Q(\Symp(D)))$
contains the image
$\Gamma_n(Q(B_n))$ of the Gambaudo-Ghys construction.
Thus the image $i^*(EH_b^2(\Symp(D)))$
also contains $\Gamma_n(EH_b^2(B_n))$.
Since $\Gamma_n(EH_b^2(B_n))$ is infinite-dimensional,
the following holds.

\begin{corollary}
  The exact bounded cohomology $EH_b^2(\Symp(D))$ and
  therefore $H_b^2(\Symp(D))$ are infinite-dimensional.
\end{corollary}

\begin{remark}
  In $Q(B_n)$, there is the abelianization homomorphism
  $a:B_n \to \mathbb{Z}$.
  It is known that $\Gamma_n(a)$ is equal to the Calabi invariant
  up to constant multiple.
  Thus, by the above argument, we can show that the
  Calabi invariant extend to a
  homogeneous quasi-morphism on $\Symp(D)$.
  This extendability of the Calabi invariant is shown in
  \cite{maruyama19}.
\end{remark}

\section{Homomorphisms on $\Symp(D)$}
  In this section, we deal with homomorphisms from $\Symp(D)$
  to $\mathbb{R}$,
  which relate to the Calabi invariant.
  It is known that the restriction map $\Symp(D) \to \Diff_{+}(S^1)$
  is surjective (see \cite{tsuboi00}).
  Put $\Symp(D)_{\rel} = \Ker(\Symp(D) \to \Diff_{+}(S^1))$.
  Let us consider the short exact sequence
  \[
    1 \to \Symp(D)_{\rel} \to \Symp(D) \to \Diff_{+}(S^1) \to 1.
  \]
  On the group $\Symp(D)_{\rel}$, there is a surjective homomorphism
  $\Cal : \Symp(D)_{\rel} \to \mathbb{R}$ called
  the Calabi invariant.
  Let us consider a part of five-term exact sequence
  \begin{align}\label{five-term}
    H^1(\Symp(D);\mathbb{R}) \to
    H^1(\Symp(D)_{\rel};\mathbb{R})^{\Symp(D)} \xrightarrow{\delta}
    H^2(\Diff_{+}(S^1);\mathbb{R}),
  \end{align}
  where $H^1(\Symp(D)_{\rel};\mathbb{R})^{\Symp(D)}$ is
  $\Symp(D)$-invariant homomorphisms
  on $\Symp(D)_{\rel}$.
  Then it is shown in \cite{bowden11} that the element
  $\delta(\Cal)$ is equal to the real Euler class
  $e_{\mathbb{R}}$ of
  $\Diff_{+}(S^1)$
  up to non-zero constant multiple.
  Thus, by the exactness of (\ref{five-term}),
  the Calabi invariant cannot extend to a homomorphism on $\Symp(D)$.
  Let us normalize the Calabi invariant as
  $\delta(\Cal) = e_{\mathbb{R}}$.
  Let $f : \mathbb{R} \to \mathbb{R}$ be a
  discontinuous homomorphism
  satisfying $f(q) = 0$ for any $q \in \mathbb{Q}$.
  Since the Calabi invariant is surjective to $\mathbb{R}$,
  the composition
  $f \circ \Cal$ is non-trivial if $f \neq 0$.
  \begin{theorem}
    The composition $f\circ \Cal$ extends to a homomorphism on $\Symp(D)$.
    In particular, the cohomology $H^1(\Symp(D))$ is infinite-dimensional.
  \end{theorem}

  \begin{proof}
    By the exactness of (\ref{five-term}),
    we have to show that the image $\delta(f\circ\Cal)$ is equal
    to $0$ in $H^2(\Diff_{+}(S^1);\mathbb{R})$.
    For any group $\Gamma$,
    let $f_* : H^*(\Gamma; \mathbb{R}) \to H^*(\Gamma; \mathbb{R})$
    denote the coefficients change by $f$.
    Since the five-term exact sequence is natural with respect to
    coefficients changes,
    we obtain
    \[
      \delta(f\circ \Cal) = \delta(f_*(\Cal))
      = f_*e_{\mathbb{R}}.
    \]
    Let $\iota :\mathbb{Z} \to \mathbb{R}$ be the inclusion
    and $e_{\mathbb{Z}} \in H^2(\Diff_{+}(S^1);\mathbb{Z})$
    the integral Euler class of $\Diff_{+}(S^1)$.
    Since the real Euler class $e_{\mathbb{R}}$ is equal
    to $\iota_{*}(e_{\mathbb{Z}})$,
    we have
    \[
      f_*e_{\mathbb{R}} =
      (f\iota)_{*}e_{\mathbb{Z}}
      = (0)_* e_{\mathbb{Z}}
      = 0
    \]
    and the theorem follows.
  \end{proof}

\bibliographystyle{amsplain}
\bibliography{qm.bib}
\end{document}